\documentclass [12pt]{article}
\usepackage{amssymb,amsmath,amsthm}

\def\N{\mathbb{N}}

\def\C{\mathbb{C}}
\def\F{\mathbb{F}}
\def\Q{\mathbb{Q}}

\newtheorem{theorem}{Theorem}[section]
\newtheorem{proposition}[theorem]{Proposition}
\newtheorem{corollary}[theorem]{Corollary}
\newtheorem{lemma}[theorem]{Lemma}
\newtheorem{definition}[theorem]{Definition}

\begin{document}
\title{On splitting perfect polynomials over $\F_{p^2}$}
 \author{Luis H. Gallardo - Olivier Rahavandrainy \\
Department of Mathematics, University of Brest,\\
6, Avenue Le Gorgeu, C.S. 93837, 29238 Brest Cedex 3, France.\\
 e-mail : luisgall@univ-brest.fr - rahavand@univ-brest.fr}
\maketitle
\newpage
\def\cB{{\cal B}}
\def\cI{{\cal I}}
\def\cJ{{\cal J}}
\def\cP{{\cal P}}
\def\cQ{{\cal Q}}
\def\cZ{{\cal Z}}
\def\cU{{\cal U}}
\def\cX{{\cal X}}
\def\cT{{\cal T}}
\def\dis{\displaystyle}
\def\sgn{{\rm sgn}}~\\
\\
{\bf{Abstract.}} We study some properties of the exponents of the
terms appearing in the splitting perfect polynomials over
$\F_{p^2}$, where $p$ is a prime number. This generalizes the work
of Beard et al. over $\F_p$.  Corrected paper. Older Lemmas 2.17 to 2.20
in published version required a fixing done herein.

\section{Introduction.}
Let $p$ be a prime number and let $\F_q$ be a finite field of
characteristic $p$ with $q$ elements. Let $A \in \F_q[x]$ be a monic
polynomial. Let $\omega(A)$ denote the number of distinct monic
irreducible factors of $A$ over $\F_q$, and let $\sigma(A)$ denote
the sum of all monic divisors of $A$ ($\sigma$ is a multiplicative
function). If $A$ divides $\sigma(A)$ (so that $\sigma(A) = A$),
then we say that $A$ is a perfect polynomial. E.~F.~Canaday, the
first doctoral student of Leonard Carlitz, began in $1941$ the study
of perfect polynomials by working over the ground field $\F_2$
\cite{Canaday}. Later, in the seventies, J. T. B. Beard Jr. et al.
extended this work in several directions (see e.g. \cite{Beard2, Beard}).
Recently, we became interested in this subject
\cite{Gall-Rahav, Gall-Rahav3, Gall-Rahav4, Gall-Rahav5}. In our two first papers, we considered the
smallest nontrivial field extension of the ground field, namely
$\F_4$, while in the two others, we continued to work on the binary
case, by considering ``odd'' and ``even'' perfect polynomials. We
began to study the special case where the polynomial splits over
$\F_q$. Our first results about splitting perfect polynomials are in
\cite{Gall-Rahav2}, where $\F_q = \F_{p^p}$ is the Artin-Schreier
extension of $\F_p$. See also \cite{Gall-Pollack-Rahav} for another
direction.

Beard et al. \cite[Theorem 7]{Beard} showed that if a perfect
monic polynomial $A$ splits over $\F_q$, then the integer
$\omega(A)$ is a multiple of $p$, and $A$ may be written as a
product:
$$A = \displaystyle{A_0 \cdots A_r},$$
where:
$$\left\{\begin{array}{l}
\displaystyle{A_i = \prod_{j\in \F_p} (x-a_i-j)^{N_{ij} p^{n_{ij}} -
1}, \ r =
\frac{\omega(A)}{p} - 1}, \ a_0~=~0, \ a_i \in~\F,\\
a_i~-~a_l~\notin~\F_p \mbox{ for } i \not= l, \ N_{ij} \ | \ q -1, \
n_{ij} \geq 0.
\end{array} \right.$$ 
We say that a polynomial $A \in \F_q[x]$ is a {\it splitting perfect} polynomial
if $A$ has all its roots in $\F_q$ and $A$ is a perfect polynomial.
We say that $A$ is {\it trivially perfect}  if for
any $0 \leq i \leq r$, $A_i$ is perfect. In that case, $A$ is
perfect and for any $0 \leq i \leq r$, there exist $N_i, n_i \in \N$
such that: $$N_{ij} = N_i, \ n_{ij} = n_i \mbox{ for all } j \in
\F_p, \ N_i \ | \ p-1.$$ The case when $q=p$ was considered by Beard
\cite{Beard2} and Beard et al. \cite{Beard}. They showed that a
polynomial
$$A = \displaystyle{\prod_{\gamma \in \F_p} (x-\gamma)^{N(\gamma)
p^{n(\gamma)} -1}}$$ is perfect over $\F_p$ if and only if the
following condition $(\star)$ holds:
$$\mbox{There exist } N, \ n \in \N \mbox{ such that } N \  | \ p-1,
\ \mbox{$N(\gamma) = N, \ n(\gamma) = n, \ \forall \ \gamma \in
\F_p$}.$$ Thus, the only splitting perfect polynomials over $\F_p$
are of the form
$$A = (x^p -x)^{Np^n -1}, \ \mbox{where } N \  | \ p-1 \mbox{ and }
n \in \N.$$
Their method consists of showing, at a first step, that $n(\gamma) =
n(\delta)$ for any $\gamma, \delta \in \F_p$, and at a second step,
that $N(\gamma) = N(\delta)$ for any $\gamma, \delta
\in \F_p$.\\
\\
If $\F_q$ is a nontrivial extension field of $\F_p$, then the
condition $(\star)$ remains sufficient (see again \cite{Beard2, Beard})
but no more necessary (see \cite[Theorem 3.4]{Gall-Rahav}, 
in the case $p=2, \ q=4$). \\
\\
If $A = \displaystyle{\prod_{\gamma \in \F_q} (x-\gamma)^{N(\gamma)
p^{n(\gamma)} - 1}}$ is perfect, then two natural cases arise:
$$\begin{array}{l}
\mbox{Case1:}\\
 \mbox{There exists $N \in \N$
such that: $N \ | \ q-1, \ N(\gamma) = N, \ \forall \ \gamma \in \F_q$.}$$\\
\\
\mbox{Case2:}\\
\mbox{There exists $n \in \N$ such that $n(\gamma) = n, \ \forall \
\gamma \in \F_q$.}
\end{array}$$
We observe that Case2 does not imply Case1 (consider trivially
perfect polynomials).\\
Let us fix an algebraic closure of $\F_p$. In order to get some
progress in the classification of splitting perfect polynomials over
a nontrivial extension field of $\F_p$, we would like to know if
Case1 implies Case2, when we work over the smallest nontrivial
extension field of $\F_p$, namely the quadratic extension $\F_{p^2}$.\\
\\
In the rest of the paper, we put $q = p^2$. Our new idea is to
consider suitable (block) circulant matrices (see \cite[Sections 5.6 and 5.8]{Davis}
). The object of this paper is to prove the
following result:
\begin{theorem} \label{principaltheorem1}~\\
Let $N \in \N$ be a divisor of $q-1$, and let 
$$A =
\displaystyle{\prod_{\gamma \in \F_q} (x-\gamma)^{N p^{n(\gamma)} -
1}}
$$
be a splitting perfect polynomial over $\F_q$.
\begin{itemize}
\item[i)]
If $N$ divides $p-1$, then $A$ is trivially perfect, so that the
integers $n(\gamma)$ may differ.
\item[ii)]
 If $N$ does not divide $p-1$, then $n(\gamma) = n(\delta):= n\;
(\text{say})$, for any $\gamma, \delta \in \F_q$, so that: $A =
(x^q-x)^{Np^n - 1}$.
\end{itemize}
\end{theorem}

\section{Proof of Theorem \ref{principaltheorem1}.}
We need to introduce some notation. The integers $0, 1, \ldots,
p-1$ will
be also considered as elements of $\F_p$.\\
We put: $$\begin{array}{l} \mbox{$A = \displaystyle{\prod_{\gamma
\in \F_q} (x-\gamma)^{N
p^{n(\gamma)} - 1}}$, where $N$ divides $q-1$},\\
\mbox{$U = \{0, 1, \ldots, p-1\} \subset \N$.} \end{array}$$ If $N
\geq 2$, we denote by $\zeta_2, \ldots, \zeta_N \in \F_q$ the $N$-th
roots of $1$,
distinct from $1$. \\
Finally, we denote by $\overline{\F_p}$ a fixed
algebraic closure of $\F_p$.
\subsection{Preliminary}\label{preliminaire}
We put: $\F_q = \F_{p^2} =  \{j_0 \alpha + j_1 \, : \, j_0, j_1 \in
\F_p \} = \F_p[\alpha]$, where $\alpha \in \overline{\F_p}$ is a
root of an irreducible polynomial of degree $2$ over $\F_p$. Every
element $i \alpha + j \in \F_q$ will be, if necessary, identified to
the pair $(i,j) \in \F_p \times \F_p$.
We define the two following order relations: \\
- on $\F_p$: $0 \leq 1 \leq 2 \leq \cdots \leq p-1$,\\
- on $\F_q$ (lexicographic order): $$(j_0,j_1) \leq (l_0,l_1) \
\mbox{ if: either
 } (j_0 < l_0) \mbox{ or } (j_0 = l_0, \ j_1 \leq l_{1}).$$
For $\gamma \in \F_q$, we put: $$\Lambda^{\gamma} = \{\delta \in
\F_q \, : \, \delta \not= \gamma, \ (\gamma + 1 - \delta)^{N} = 1\} =
\{\gamma + 1 - \zeta_2, \ldots, \gamma + 1 - \zeta_N \}.$$ Observe
that:
$$\begin{array}{l} \mbox{$\Lambda^{\gamma} \not= \emptyset \ $
if $\ N \geq 2\ $ and $\ \Lambda^{\gamma} \subset \{\gamma + j \ : \
j \in \F_p \}$ if $N \ | \ p-1$}. \end{array}$$
For $P, Q \in \F_q[x]$, $P^{m} \ || \ Q$ means that $P^{m}$ divides
$Q$ and that $P^{m+1}$ does not divide $Q$. The following
straightforward result is useful:

\begin{lemma} \rm{(Lemma 2 in \cite{Beard})} \label{lemma2beard}~\\
The polynomial $A$ is perfect if and only if for any irreducible
polynomial $P \in \F_q[x]$, and for any positive integers $m_1,
m_2$, we have:$$(P^{m_1} \ || \ A, \ P^{m_2} \ || \
\sigma(A))~\Rightarrow~(m_1~=~m_2).$$
\end{lemma}

We obtain an immediate consequence:
\begin{proposition} \label{perfectcritere}~\\
If $N \geq 2$, then the polynomial $A$ is perfect if and only if:
$$N p^{n(\gamma +1)} = p^{n(\gamma)} + \displaystyle{\sum_{\delta
\in \Lambda^{\gamma}} p^{n(\delta)}}, \ \forall \gamma~\in~\F_q.$$
\end{proposition}
{\bf{Proof}}:\\
For every $\gamma \in \F_q$, we may apply Lemma \ref{lemma2beard} to
the polynomial $P = x-\gamma-1$, where $m_1 = N p^{n(\gamma +1)} - 1
\geq 1$ since $N \geq 2$. \\
By considering: $$\sigma(A) = \displaystyle{\prod_{\delta \in \F_q}
\sigma((x-\delta)^{N p^{n(\delta)} - 1}) = \prod_{\delta \in
\F_q}\mbox{{\Large{(}}}(x-\delta-1)^{p^{n(\delta)} - 1}
\prod_{j=2}^{N} (x-\delta -
\zeta_j)^{p^{n(\delta)}}\mbox{{\Large{)}}}},$$ we see that the
exponent of $P$ in $\sigma(A)$ is exactly the integer:$$m_2 =
p^{n(\gamma)} - 1 + \displaystyle{\sum_{\delta \in \Lambda^{\gamma}}
p^{n(\delta)}}.$$ Furthermore, $m_2 \geq 1$ since $\Lambda^{\gamma}$
is not empty.
\begin{flushright} $\Box$
\end{flushright}

\subsection{Circulant matrices.} \label{circulant} In this section, we
recall some results about circulant matrices and block circulant
matrices (see \cite[Chapters $3$ and $4$]{Davis}), that will be
useful in the proof of our main result.

\begin{definition} \label{circulant}
Let $n$ be a positive integer. A  {\em circulant} matrix of order $n$ is a
square matrix $C = (c_i^j)_{0 \leq i,j \leq n-1}$ such that the
entries $c_i^j$ satisfy:
$$c_i^{j} = c_{i-1}^{j-1}, \
c_i^{0} = c_{i-1}^{n-1}, \mbox{ for } 1 \leq i,j \leq n-1.$$
\end{definition}
\begin{definition} \label{circulantbloc}
Let $n, m$ be positive integers. A  {\em block circulant} matrix of type
$(n,m)$ is a square matrix, of order $nm$: $S = (S_i^j)_{0 \leq i,j
\leq n-1}$ such that:
$$\left\{\begin{array}{l}
\mbox{each matrix $S_i^j$ is a square matrix of order $m$,}\\
S_i^{j} = S_{i-1}^{j-1}, \ S_i^{0} = S_{i-1}^{n-1}, \mbox{ for } 1
\leq i,j \leq n-1.
\end{array}
\right.$$ Furthermore, if every $S_i^j$ is a circulant matrix, then
$S$ is called a {\em block circulant with circulant blocks}.
\end{definition}
{\bf{Notation}}\\
- If $C$ is a circulant matrix of order $n$ and if we denote,  for
$0 \leq j \leq n-1$:
$$c_j = c_0^j,$$
then $C$ may be written as:
$$\displaystyle{C = {\rm{circ}}(c_0,
\ldots, c_{n-1}) = \left(\begin{array}{cccc} c_0&c_1&...&c_{n-1}\\
c_{n-1}&c_0&...&c_{n-2}\\
\vdots&\vdots&\vdots&\vdots\\
c_1&c_2&...&c_{0}
\end{array}\right)}.$$
- Analogously, a block circulant matrix $S$ may be written as:
$$\displaystyle{S = {\rm{bcirc}}(S_0,
\ldots, S_{n-1}) = \left(\begin{array}{cccc} S_0&S_1&...&S_{n-1}\\
S_{n-1}&S_0&...&S_{n-2}\\
\vdots&\vdots&\vdots&\vdots\\
S_1&S_2&...&S_{0}
\end{array}\right)},$$
where $S_j = S_0^j$, for $0 \leq j \leq n-1$.\\
\\
We shall use several times the following crucial result when $n =
p$.
\begin{lemma} {\rm{(see \cite[Section 3.2]{Davis})}}
\label{diagocirc}~\\
Let $n$ be a positive integer. Any circulant matrix $C =
{\rm{circ}}(c_0, \ldots, c_{n-1})$ is diagonalizable on $\C$, and
admits the following eigenvalues:
$$\displaystyle{c_0 + c_1 \omega ^k+ \cdots + c_{n-1}
(\omega^{k})^{n-1} = \sum_{l=0}^{n-1} c_{l} (\omega^{k})^l}, \mbox{
for } k \in \{0, \ldots, n-1\},$$ where:
$$\omega = \cos(2 \pi/n) + i \sin(2 \pi/n) \in \C,$$ is a $n$-{\rm{th}} primitive
root of unity.
\end{lemma}
\begin{lemma} {\rm{(see \cite[Theorem 5.8.1]{Davis})}}
\label{diagosimultane}~\\
Let $n$ be a positive integer and let $S = {\rm{bcirc}}(S_0, \ldots,
S_{n-1})$ be a block circulant of type $(n,n)$, with circulant
blocks, then $S_0, \ldots, S_{n-1}$
are~simultaneously~diagonalizable~on~$~\C$.
\end{lemma}

\subsection{The proof.}
For $\gamma \in \F_q$, we put: $x_{\gamma} = p^{n(\gamma)}$. If we
identify $\gamma = i \alpha + j$ and $\delta = r \alpha + s$ to the
pairs $(i,j), \ (r,s) \in {\F_p}^2$, we may order the unknowns
$x_{ij}$ and $x_{rs}$, as follows:
$$x_{ij} \leq x_{rs} \ \iff \ (i,j) \leq (r,s),$$
according to the order relation on $\F_q$ defined in Section
\ref{preliminaire}. We obtain, from Proposition
\ref{perfectcritere}, a linear system of $q$ equations in $q$
unknowns: the $x_{\gamma}$'s: \begin{equation} \label{systeme}
Nx_{\gamma +1} = x_{\gamma} + \displaystyle{\sum_{\delta \in
\Lambda^{\gamma}} x_{\delta}}, \ \gamma \in \F_q.
\end{equation}
We denote by $S$ the matrix of the linear system (\ref{systeme}). For $i,j \in
\F_p$, we denote by $S_i^j$ the square matrix of order $p$
corresponding to the coefficients of unknowns $x_{j \alpha}, x_{j
\alpha + 1} \ldots, x_{j \alpha+p-1}$, in the $p$ equations:
$$N x_{\gamma +1} = x_{\gamma} + \displaystyle{\sum_{\delta \in
\Lambda^{\gamma}} x_{\delta}}, \mbox{ where } \gamma \in \{i \alpha
, i \alpha + 1, \ldots, i \alpha + p-1 \}.$$ We have, by direct
computations, the following results:
\begin{lemma}~\\
The matrix $S$ can be written as a block matrix:
$$S =
(S_i^j)_{0 \leq i,j \leq p-1} = \left(\begin{array}{lcl}
S_0^0&...&S_0^{p-1}\\
\vdots&\vdots&\vdots\\
...& S_i^j&...\\
\vdots&\vdots&\vdots\\
S_{p-1}^0&...&S_{p-1}^{p-1}
\end{array}\right).$$
\end{lemma}

\begin{lemma} \label{elementdeSij}~\\
If $(e_i^j)_{mn}$ is the entry in row $m$ and column $n$ of $S_i^j$,
for $0 \leq  m,n \leq p-1$, then:
$$\left\{\begin{array}{l}
(e_i^j)_{mn} = 1 \mbox{ if either } (j \alpha + n = i \alpha + m)
\mbox{ or } (j \alpha + n \in \Lambda^{i \alpha + m}),\\
(e_i^j)_{mn} = -N \mbox{ if } j \alpha + n = i \alpha + m + 1,\\
(e_i^j)_{mn} = 0 \mbox{ otherwise}.
\end{array}
\right.$$
\end{lemma}
By Lemma \ref{elementdeSij}, and from the definition of
$\Lambda^{\gamma}$, for $\gamma \in \F_q$, we obtain:
\begin{lemma}
$$\left\{\begin{array}{l}
(e_i^j)_{mn} = 1 \mbox{ if } ((i-j) \alpha +m-n+1)^N = 1,\\
(e_i^j)_{mn} = -N \mbox{ if } (i = j \mbox{ and } n = m+1),\\
(e_i^j)_{mn} = 0, \mbox{ otherwise.}
\end{array}
\right.$$
\end{lemma}
It follows that:
\begin{lemma} \label{bloccirculant}
$$\left\{\begin{array}{l} S_i^{j} = S_{i-1}^{j-1}, \
S_i^{0} = S_{i-1}^{p-1}, \mbox{ for } 1 \leq i,j \leq p-1,\\
(e_0^j)_{mn} = (e_0^j)_{m-1 \ n-1}, \ (e_0^j)_{m0} = (e_0^j)_{m-1 \
p-1}, \mbox{ for } 1 \leq j,m,n \leq p-1.
\end{array}
\right.$$
\end{lemma}
By putting: $S_0^{j} = S_j$, we deduce from Lemma
\ref{bloccirculant} the following two lemmas:
\begin{lemma} The matrix $S$ is a block
circulant matrix:
$$S = {\rm{bcirc}}(S_0, \ldots, S_{p-1}).$$
\end{lemma}
\begin{lemma}
Every matrix $S_j, \ j \in U$, is a circulant matrix of order $p$:
$$S_j =
{\rm{circ}}((e_0^j)_{00}, \ldots, (e_0^j)_{0p-1}).$$
\end{lemma}
In the following, for $i, j \in \{0, \ldots, p-1\}$, we put:
$$\begin{array}{l}
a_{j,i} = (e_0^j)_{0i}, \mbox{ (the entry in row $0$ and column $i$
of $S_j$).}
\end{array}$$
Thus, the matrix $S_j$ becomes:
$$S_j = {\rm{circ}}(a_{j,0}, \ldots, a_{j,p-1}).$$ We immediately obtain:

\begin{lemma} \label{Nnedivisepas}~
\begin{itemize}
\item[ i)]
$ \displaystyle{a_{0,0} = 1, \ a_{0,1} = -N,} \
\displaystyle{\ a_{j,i} \in \{0, 1\} {\mbox{ if }} (j,i) \not\in
\{(0,0), (0,1)\}}$
\item[ ii)]
$\displaystyle{\sum_{(i,j) \in U^2} a_{j,i} = 0}$ 
\item[ iii)]
$N$ divides $p-1$ if and
only if $S_j = 0$ for any $j \in U \setminus \{0\}$
\item[iv)]
If $N = q-1$, then $a_{j,i} = 1$ for any $(j,i) \not= (0,1)$.
\end{itemize}
\end{lemma}

{\bf{Proof}}:\\
We consider the equation corresponding to $\gamma = 0 = (0,0)$,
in the linear system (\ref{systeme}).\\
The part i) is obtained by direct computations.\\
ii) we obtain:
$$\displaystyle{\sum_{(i,j) \in U^2} a_{j,i} = a_{0,0}
+ a_{0,1} + \sum_{\delta \in \Lambda^{0}} 1 = 1 - N + {\rm{card}}
(\Lambda^{0}) = 0},$$ since $\Lambda^{\gamma}$ contains
exactly $N-1$ elements, for any $\gamma \in \F_q$.\\
iii) if $N$ divides $p-1$ and if $j \not= 0$, then, for any $i \in
\F_p$: $$a_{j,i} \not= -N, \mbox{ and } a_{j,i} \not= 1 \mbox{ since
} ((0-j) \alpha + 0-i+1)^N \not= 1.$$
Thus, $S_j = 0$.\\
Conversely, if $a_{j,i} = 0$ for any $i, j \in \F_p$ such that $j
\not= 0$, then $a_{j,i} \not= 1$ for any such $i, j$. By the same
arguments, we see also that $N$ must divide $p-1$.\\
iv) it follows by the fact: $\Lambda^0 = \F_q \setminus \{0, 1\}$ if
$N = q-1$.
\begin{flushright} $\Box$
\end{flushright}

\begin{lemma} \label{Ndivise}~\\
If $N$ divides $p-1$, then $S$ is the block diagonal matrix: $$S =
{\rm{diag}}(S_0, \ldots, S_0) = \left(\begin{array}{lccl}
S_0&0&...&0\\
0&S_0&...&0\\
\vdots&\vdots&\vdots&\vdots\\
0&0&...&S_0
\end{array}\right) .$$
\end{lemma}
{\bf{Proof}}:\\
By Lemma \ref{Nnedivisepas}- iii), $S_j = 0$ for all $j \in U
\setminus \{0\}$, so that $S = {\rm{diag}}(S_0, \ldots, S_0)$.
\begin{flushright} $\Box$
\end{flushright}
We put: $$\begin{array}{l} \omega = \cos(2 \pi/p) + i \sin(2
\pi/p)~\in~\C,\\
\displaystyle{\lambda_{j,k} = \sum_{l=0}^{p-1} a_{j,l}
(\omega^{k})^l}, \mbox{ for } j, k \in U\\
\mbox{$\Delta_{j} = {\rm{diag}}(\lambda_{j,0}, \ldots, \lambda_{j,p-1})$,
for $j \in U$},\\
\Delta = {\rm{bcirc}}(\Delta_{0}, \ldots, \Delta_{p-1}).
\end{array}$$

We obtain the
\begin{proposition}~\\
The matrices $S$ and $\Delta$ have the same rank.
\end{proposition}
{\bf{Proof:}} By Lemma \ref{diagocirc}, for each $j \in U$, the
matrix $S_j$ is diagonalizable and $\lambda_{j,0}, \ \ldots,
\lambda_{j,p-1}$ are its eigenvalues. Furthermore, by Lemma
\ref{diagosimultane}, the matrices $S_j, \ j \in U$, are
simultaneously diagonalizable. So, the matrices $S$ and $\Delta$ are
similar. We are done.
\begin{flushright} $\Box$
\end{flushright}

Now, if we put together the rows: $$L_l, L_{p+l}, L_{2p+l}, \ldots,
L_{(p-1)p+l},$$ of the matrix $\Delta$, for each integer $l~\in~\{0,
\ldots, p-1\}$,
we obtain a matrix $\Delta'$, with the same rank. \\
By putting together, for each integer $l \in \{0, \ldots,p-1\}$,
the columns:
$$C_l, C_{p+l}, C_{2p+l}, \ldots, C_{(p-1)p+l}$$ of the matrix
$\Delta'$, we obtain a matrix $\tilde{\Delta}$, which has also the
same rank as $\Delta$. The matrix $\tilde{\Delta}$ is a block
diagonal matrix: $$\tilde{\Delta} =
 {\rm{diag}}(\tilde{\Delta}_{0}, \ldots, \tilde{\Delta}_{p-1}),$$
where $$\tilde{\Delta}_{k} = {\rm{circ}}(\lambda_{0,k}, \ldots,
\lambda_{p-1,k})$$ is a circulant matrix, for any $k \in U$. Thus,
we obtain the

\begin{proposition}~\\
The matrices $S$ and $\tilde{\Delta}$ have the same rank.
\end{proposition}~\\

To finish the proof of Theorem \ref{principaltheorem1}, we need the
following results.

\begin{lemma} \label{cyclotomic}
Let $j \in U \setminus \{0\}$ and $u_0, \ldots, u_{p-1} \in \Q$ such
that $\displaystyle{\sum_{r \in U} u_{r} (\omega^{j})^r = 0}$, then:
$$u_r = u_0,  \ \forall r \in U.$$
\end{lemma}

\begin{proof}
Since $\{1, \omega^j, \ldots, (\omega^{j})^{p-1}\} = \{1, \omega,
\ldots, \omega^{p-1}\}$, we may assume that $j=1$. It suffices to
observe that the cyclotomic polynomial $\Phi_p(x) = 1~+~\cdots
~+~x^{p-1}$, which is irreducible, is the minimal
polynomial~of~$\omega$.
\end{proof}

\begin{lemma} \label{rangdeS0}
The matrix $S_0$ has rank $p-1$.
\end{lemma}

\begin{proof}
By
Lemma 2.5
, the eigenvalues of the matrix $S_0$ are:
$$\left\{\begin{array}{l}
\displaystyle{\nu_0 = a_{0,0} + \cdots + a_{0,p-1} =
\sum_{(j,i) \in U^2} a_{j,i} = 0},\\
\displaystyle{\nu_l = \sum_{r \in U} a_{0,r} (\omega^l)^r}, \mbox{
for } l \in U \setminus \{0\}.
\end{array}
\right.$$ If $\nu_l = 0$ for some $l \in U \setminus \{0\}$, then by
Lemma \ref{cyclotomic}, we have:
$$a_{0,r} = a_{0,0}
\ \forall r \in U.$$ It is impossible since $a_{0,0} = 1$ and
$a_{0,1} = -N$. Thus, $S_0$ has exactly $p-1$ nonzero eigenvalues.
We are done.
\end{proof}

If $N$ does not divide $p-1$, the following two lemmas give the rank
of $\tilde{\Delta}_{k}$ for $k \in U$.

\begin{lemma}
If $N$ does not divide $p-1$, then the matrix $\tilde{\Delta}_{0}$
has rank $p-1$.
\end{lemma}

\begin{proof}
We know, by 
Lemma 2.5 that $\tilde{\Delta}_{0}$ has the
following eigenvalues: $$\left\{\begin{array}{l} \displaystyle{\mu_0
= \lambda_{0,0} + \cdots + \lambda_{p-1,0} =
\sum_{(j,i) \in U^2} a_{j,i} = 0},\\
\displaystyle{\mu_l = \sum_{r \in U} \lambda_{r,0} (\omega^l)^r},
\mbox{ for } l \in U \setminus \{0\}.
\end{array}
\right.$$ If $\mu_l = 0$ for some $l \in U \setminus \{0\}$, then by
Lemma \ref{cyclotomic}, we have:
$$\lambda_{r,0} = \lambda_{0,0}.$$
We obtain: $$0 = \displaystyle{\sum_{(r,s) \in U^2} a_{r,s} =
\sum_{r \in U} a_{r,0} + \cdots + \sum_{r \in U} a_{r,p-1} = \sum_{r
\in U} \lambda_{r,0} = p \ \lambda_{0,0}}.$$ Hence $\lambda_{0,0}
=0,$ and $$\displaystyle{\sum_{r \in U} a_{0,r} = \lambda_{0,0} = 0
= \sum_{(i,j) \in U^2} a_{j,i} = \sum_{r \in U} a_{0,r} +
\sum_{(j,r) \in U^2, \ j \not= 0} a_{j,r}}.$$ It follows that
$a_{j,r}~=~0$ for any $j,r \in U$ such that $j \geq 1$. It is
impossible since the matrix $S_j$ is not the zero matrix by Lemma
\ref{Nnedivisepas}- ii).
\end{proof}

\begin{lemma}
If $N$ does not divide $p-1$, then for any $j \in U \setminus
\{0\}$, the matrix $\tilde{\Delta}_j$ has~rank~$p$.
\end{lemma}

\begin{proof}
By
Lemma 2.5, the matrix $\tilde{\Delta}_j$ has the
following eigenvalues:
$$\begin{array}{l} \displaystyle{\mu_{jl} = \sum_{s \in U}\lambda_{s,j}
(\omega^l)^{s} = \sum_{(r,s) \in U^2} a_{s,r} \omega^{rj+sl}}, \ l
\in U.
\end{array}$$
For $t \in U$, we put: $U_t = \{(r,s) \in U^2 \, : \, sj+rl \equiv t
\ {\rm{mod}} \ p \}.$ The set $U^2$ is the disjoint union $U_0
\bigsqcup \cdots \bigsqcup U_{p-1}$. So, we can write:
$$\displaystyle{\mu_{jl} = \sum_{(r,s) \in U^2} a_{r,s}
\omega^{sj+rl} = \sum_{t \in U} (\sum_{(r,s) \in U_t} a_{r,s})
\omega^t}.$$ If $\mu_{jl} = 0$, then by Lemma \ref{cyclotomic}, we
have: $$\displaystyle{\sum_{(r,s) \in U_t} a_{r,s} = \sum_{(r,s) \in
U_0} a_{r,s}}, \ \forall t \in U.$$ Thus, $$0 =
\displaystyle{\sum_{(r,s) \in U^2} a_{r,s} = \sum_{t \in U}
\sum_{(r,s) \in U_t} a_{r,s} = p \ \sum_{(r,s) \in U_0} a_{r,s}}.$$
So,
$$\displaystyle{\sum_{(r,s) \in U_0} a_{r,s} = 0}.$$ Furthermore,
$a_{r,s} \geq 0$ for any $(r,s) \in U_0$ since $(0,1) \not\in U_0$.
Thus: $$0 = \displaystyle{\sum_{(r,s) \in U_0} a_{r,s} \geq a_{0,0}
= 1}.$$ It is impossible.
\end{proof}

We obtain our main results:
\begin{corollary} \label{casNdivise}~\\
If $N$ divides $p-1$, then $n(\gamma) = n(\gamma+j)$ for any $\gamma
\in \F_q$, $j \in \F_p$.
\end{corollary}
{\bf{Proof}}:\\
By Lemma \ref{Ndivise}, the matrix $S$ is exactly the diagonal
matrix: ${\rm{diag}}(S_0, \ldots, S_0)$, so the linear system
(\ref{systeme}) splits into $p$ linear systems (each of which is of
matrix $S_0$)
 in $p$ unknowns: $x_{\gamma}, x_{\gamma +1}, \ldots, x_{\gamma
+p-1}$: \begin{equation} \label{systeme2} N x_{\gamma +j+1} =
x_{\gamma+j} + \displaystyle{\sum_{\delta \in \Lambda^{\gamma+j}}
x_{\delta}},\mbox{ for } \gamma = i \alpha, \ i, j \in \F_p.
\end{equation}
Moreover, by Lemma \ref{rangdeS0}, $S_0$ has rank $p-1$. It remains
to observe that $(1, \ldots, 1)$ belongs to the kernel of $S_0$,
since
$$a_{0,0} + \cdots + a_{0,p-1} =
\displaystyle{\sum_{(i,j) \in U^2} a_{j,i} = 0,}$$ by Lemma
\ref{Nnedivisepas} ii).
\begin{flushright} $\Box$
\end{flushright}
\begin{corollary}~\\
If $N$ does not divide $p-1$, then $n(\gamma) = n(\delta)$ for any
$\gamma, \delta \in \F_q$.
\end{corollary}
{\bf{Proof}}:\\
In that case, the matrix $\tilde{\Delta}$ (and thus the matrix $S$)
has rank: $$p-1+ (p - 1)p = p^2 - 1 = q-1.$$ Moreover, $(1, \ldots,
1)$ belongs to the kernel of $S$, since
$$\displaystyle{\sum_{(i,j) \in U^2} a_{j,i} = 0,}$$ by Lemma
\ref{Nnedivisepas} ii). So, we are done.
\begin{flushright} $\Box$
\end{flushright}

{\bf{Final remarks}}\\
1) If $q = p^m$, for $m \geq 3$, then our method fails since we
cannot apply~Lemma~\ref{diagosimultane}.\\
2) If $p = 2$, then the splitting perfect polynomials
over $\F_4$ are known (see \cite[Theorem 3.4]{Gall-Rahav}).\\
3) By using a computer program, we obtain a complete list of perfect
polynomial over $\F_9$, of the form:
$$\mbox{$\displaystyle{\prod_{\gamma \in \F_9} (x-\gamma)^{N(\gamma) -
1}},$ where $N(\gamma) \ | \ 8$, and $n(\gamma) = 0, \ \forall \
\gamma \in \F_9$}.$$
Except for trivially perfect polynomials and
for perfect polynomials of the form:
$$\mbox{$(x^9-x)^{N-1}$, where $N \in \{1,2,4,8\}$,}$$ We obtain
two other families: $\mbox{$A_1(x-a)$ and $A_2(x-a), \ a \in
\F_9$,}$ where: 
$\alpha \in \F_9$  satisfy $\alpha^2=-1$ and
$$\begin{array}{ll} \mbox{- for } A_1(x):& \\
& N(0) = N(\alpha) = N(2\alpha) = 4,\\
& N(j) = N(\alpha+j) = N(2\alpha+j) = 2, \ j \in \{1,2\}.\\
\\
\mbox{- for } A_2(x):&\\
&N(1) = N(\alpha+1) = N(2\alpha+1) = 2, \\
&N(j) = N(\alpha+j) = N(2\alpha+j) = 4, \ j \in \{0,2\}.
\end{array}$$
\\
Then, we can deduce (see \cite{Beard2}), for a fixed positive
integer $m$, the list of all perfect polynomials of the form:
 $\displaystyle{\prod_{\gamma \in \F_9} (x-\gamma)^{N(\gamma)p^m -
1}}.$\\
The computer took some substantial time to do the job. So, we may
think that the determination of all splitting perfect polynomials
over a finite field is a non-trivial problem.
\def\thebibliography#1{\section*{\titrebibliographie}
\addcontentsline{toc}
{section}{\titrebibliographie}\list{[\arabic{enumi}]}{\settowidth
 \labelwidth{[
#1]}\leftmargin\labelwidth \advance\leftmargin\labelsep
\usecounter{enumi}}
\def\newblock{\hskip .11em plus .33em minus -.07em} \sloppy
\sfcode`\.=1000\relax}
\let\endthebibliography=\endlist

\def\biblio{\def\titrebibliographie{References}\thebibliography}
\let\endbiblio=\endthebibliography




\newbox\auteurbox
\newbox\titrebox
\newbox\titrelbox
\newbox\editeurbox
\newbox\anneebox
\newbox\anneelbox
\newbox\journalbox
\newbox\volumebox
\newbox\pagesbox
\newbox\diversbox
\newbox\collectionbox
\def\fabriquebox#1#2{\par\egroup
\setbox#1=\vbox\bgroup \leftskip=0pt \hsize=\maxdimen \noindent#2}
\def\bibref#1{\bibitem{#1}


\setbox0=\vbox\bgroup}
\def\auteur{\fabriquebox\auteurbox\styleauteur}
\def\titre{\fabriquebox\titrebox\styletitre}
\def\titrelivre{\fabriquebox\titrelbox\styletitrelivre}
\def\editeur{\fabriquebox\editeurbox\styleediteur}

\def\journal{\fabriquebox\journalbox\stylejournal}

\def\volume{\fabriquebox\volumebox\stylevolume}
\def\collection{\fabriquebox\collectionbox\stylecollection}
{\catcode`\- =\active\gdef\annee{\fabriquebox\anneebox\catcode`\-
=\active\def -{\hbox{\rm
\string-\string-}}\styleannee\ignorespaces}}
{\catcode`\-
=\active\gdef\anneelivre{\fabriquebox\anneelbox\catcode`\-=
\active\def-{\hbox{\rm \string-\string-}}\styleanneelivre}}
{\catcode`\-=\active\gdef\pages{\fabriquebox\pagesbox\catcode`\-
=\active\def -{\hbox{\rm\string-\string-}}\stylepages}}
{\catcode`\-
=\active\gdef\divers{\fabriquebox\diversbox\catcode`\-=\active
\def-{\hbox{\rm\string-\string-}}\rm}}
\def\ajoutref#1{\setbox0=\vbox{\unvbox#1\global\setbox1=
\lastbox}\unhbox1 \unskip\unskip\unpenalty}
\newif\ifpreviousitem
\global\previousitemfalse
\def\separateur{\ifpreviousitem {,\ }\fi}
\def\voidallboxes
{\setbox0=\box\auteurbox \setbox0=\box\titrebox
\setbox0=\box\titrelbox \setbox0=\box\editeurbox
\setbox0=\box\anneebox \setbox0=\box\anneelbox
\setbox0=\box\journalbox \setbox0=\box\volumebox
\setbox0=\box\pagesbox \setbox0=\box\diversbox
\setbox0=\box\collectionbox \setbox0=\null}
\def\fabriquelivre
{\ifdim\ht\auteurbox>0pt
\ajoutref\auteurbox\global\previousitemtrue\fi
\ifdim\ht\titrelbox>0pt
\separateur\ajoutref\titrelbox\global\previousitemtrue\fi
\ifdim\ht\collectionbox>0pt
\separateur\ajoutref\collectionbox\global\previousitemtrue\fi
\ifdim\ht\editeurbox>0pt
\separateur\ajoutref\editeurbox\global\previousitemtrue\fi
\ifdim\ht\anneelbox>0pt \separateur \ajoutref\anneelbox
\fi\global\previousitemfalse}
\def\fabriquearticle
{\ifdim\ht\auteurbox>0pt        \ajoutref\auteurbox
\global\previousitemtrue\fi \ifdim\ht\titrebox>0pt
\separateur\ajoutref\titrebox\global\previousitemtrue\fi
\ifdim\ht\titrelbox>0pt \separateur{\rm in}\
\ajoutref\titrelbox\global \previousitemtrue\fi
\ifdim\ht\journalbox>0pt \separateur
\ajoutref\journalbox\global\previousitemtrue\fi
\ifdim\ht\volumebox>0pt \ \ajoutref\volumebox\fi
\ifdim\ht\anneebox>0pt  \ {\rm(}\ajoutref\anneebox \rm)\fi
\ifdim\ht\pagesbox>0pt
\separateur\ajoutref\pagesbox\fi\global\previousitemfalse}
\def\fabriquedivers
{\ifdim\ht\auteurbox>0pt
\ajoutref\auteurbox\global\previousitemtrue\fi
\ifdim\ht\diversbox>0pt \separateur\ajoutref\diversbox\fi}
\def\endbibref
{\egroup \ifdim\ht\journalbox>0pt \fabriquearticle
\else\ifdim\ht\editeurbox>0pt \fabriquelivre
\else\ifdim\ht\diversbox>0pt \fabriquedivers \fi\fi\fi
.\voidallboxes}

\let\styleauteur=\sc
\let\styletitre=\it
\let\styletitrelivre=\sl
\let\stylejournal=\rm
\let\stylevolume=\bf
\let\styleannee=\rm
\let\stylepages=\rm
\let\stylecollection=\rm
\let\styleediteur=\rm
\let\styleanneelivre=\rm

\begin{biblio}{99}

\begin{bibref}{Beard2}
\auteur{J. T. B. Beard Jr}  \titre{Perfect polynomials revisited}
\journal{Publ. Math. Debrecen} \volume{38/1-2} \pages 5-12 \annee
1991
\end{bibref}

\begin{bibref}{Beard}
\auteur{J. T. B. Beard Jr, J. R. O'Connell Jr, K. I. West}
\titre{Perfect polynomials over $GF(q)$} \journal{Rend. Accad.
Lincei} \volume{62} \pages 283-291 \annee 1977
\end{bibref}

\begin{bibref}{Canaday}
\auteur{E. F. Canaday} \titre{The sum of the divisors of a
polynomial} \journal{Duke Math. J.} \volume{8} \pages 721-737 \annee
1941
\end{bibref}

\begin{bibref}{Davis}
\auteur{P. J. Davis} \titrelivre{Circulant Matrices}
\editeur{Chelsea Publishing New York, N.Y.} \anneelivre 1979
(Reprinted 1994)
\end{bibref}

\begin{bibref}{Gall-Pollack-Rahav}
\auteur{L. Gallardo, P. Pollack, O. Rahavandrainy} \titre{On a
conjecture of Beard, O'Connell and West concerning perfect
polynomials} \journal{Finite Fields Appl.} \volume{14(1)} \pages
242-249 \annee 2008
\end{bibref}

\begin{bibref}{Gall-Rahav}
\auteur{L. Gallardo, O. Rahavandrainy} \titre{On perfect polynomials
over $\F_4$} \journal{Port. Math. (N.S.)} \volume{62(1)} \pages
109-122 \annee 2005
\end{bibref}

\begin{bibref}{Gall-Rahav3}
\auteur{L. Gallardo, O. Rahavandrainy} \titre{Perfect polynomials
over $\F_4$ with less than five prime factors} \journal{Port. Math.
(N.S.) } \volume{64(1)} \pages 21-38 \annee 2007
\end{bibref}

\begin{bibref}{Gall-Rahav4}
\auteur{L. H. Gallardo, O. Rahavandrainy} \titre{Odd perfect
polynomials over $\F_2$} \journal{J. Th\'eor. Nombres Bordeaux}
\volume{19} \pages 165-174 \annee 2007
\end{bibref}

\begin{bibref}{Gall-Rahav5}
\auteur{L. H. Gallardo, O. Rahavandrainy} \titre{Even perfect
polynomials over $\F_2$ with four prime factors}
\journal{
Int. J. Pure Appl. Math.; (see also arXiv: 0712.2602);}
\volume{52, no.2} \pages 301-314
 \annee 2009
\end{bibref}

\begin{bibref}{Gall-Rahav2}
\auteur{L. H. Gallardo, O. Rahavandrainy} \titre{On splitting
perfect polynomials over $\F_{p^p}$} \journal{Preprint (2007)}
\end{bibref}

\begin{bibref}{Rudolf}
\auteur{Rudolf Lidl, Harald Niederreiter} \titrelivre{Finite Fields,
Encyclopedia of Mathematics and its Applications} \editeur{Cambridge
University Press} \anneelivre 1983 (Reprinted 1987)
\end{bibref}

\end{biblio}

\end{document}